\documentclass[11pt]{amsart}

\newcommand{\pp}{\mathcal{P}}
\newcommand{\cc}{\mathcal{C}}
\newcommand{\bb}{\mathcal{B}}
\newcommand{\hh}{\mathcal{H}}
\usepackage{mystyle}

\usepackage{tcolorbox}
\usepackage{float}

\begin{document}

\title{Strong chromatic index of bipartite graphs}
%\footnote{Dedicated to Professor Sasha Kostochka on the occasion of his retirement.}
\date{}

\author{Yanli Hao, Tianchi Yang, and Xingxing Yu}
\address{School of Mathematics, Georgia Institute of Technology, Atlanta, GA 30332}
\footnote{XY was partially supported by NSF Grant DMS--2348702}
\email{yhao98@gatech.edu, tyang439@gatech.edu, yu@math.gatech.edu}

%\author{Author 2}
%\address{University 2}
%\email{author2@university2.edu}

\begin{abstract}
An edge-coloring of a graph $G$ is called a strong edge-coloring if all its color classes are induced matchings in $G$;  
the minimum number of colors required for such a coloring, denoted by $\chi_{s}'(G)$, is known as the strong chromatic index of $G$. For each vertex $v$ of a graph $G$, let $d_G(v)$ denote the degree of $v$ in $G$. Let $G$ be a bipartite graph with partite sets $A$ and $B$, and let $\Delta_A=\max\{d_G(a): a\in A\}$ and $\Delta_B=\max\{d_G(b): b\in B\}$. A conjecture of Brualdi and Quinn Massey asserts that
\( \chi_s'(G) \le \Delta_A \Delta_B\).  
In this paper, we show that  \(\chi_s'(G) \le 1.676\, \Delta_A \Delta_B\) provided that $\Delta_A$ and $\Delta_B$ are sufficiently large. 
\end{abstract}

\maketitle

\section{Introduction}

All graphs considered in this paper are finite and simple (i.e., without loops or multiple edges). Let $G$ be a graph $G$. We use $V(G)$ and $E(G)$ to denote the vertex set and edge set of $G$, respectively, and use $\Delta(G)$ to denote the maximum degree of $G$. For any set $S\subseteq V(G)$, we use $G[S]$ to denote the subgraph of $G$ induced by $S$. For any set $T\subseteq E(G)$, $G-T$ denotes the graph obtained from $G$ by deleting the edges in $T$.  For a vertex $x\in V(G)$, we let $N_G(x)$ denote the neighborhood of $x$ in $G$ (i.e., the set of vertices adjacent to $x$) and $E_G(x)$ denote the set of edges incident with $x$, and let $d_G(x):=|N_G(x)|=|E_G(x)|$.  
The distance between two distinct edges $e$ and $f$ of $G$, denoted by $d_G(e, f)$, is the length of a shortest path in $G$ connecting $e$ and $f$ but containing neither $e$ nor $f$. (We will study pairs $\{e,f\}$ with $d_G(e,f)\le 1$, as such $\{e,f\}$ cannot be contained in an induced matching.)  Furthermore, for $X,Y\subseteq V(G)$, let $E_G(X,Y)$ denote the set of all edges of $G$ incident with both a vertex in $X$ and a vertex in $Y$. When there is no danger of confusion, we drop the subscript $G$ in the above notations.

An edge-coloring of \( G \) is called a {\it strong edge-coloring} if, for every color class $M$, $M$ is an induced matching, that is, $E(G[V(M)]) = M$. The minimum number of colors required for such a coloring, denoted by \( \chi'_s(G) \), is called the {\it strong chromatic index} of \( G \).  Erd\H{o}s and Ne\v{s}etr\v{e}l~\cite{Erdos_Nesetril_1988}  conjectured in 1988 that
\(
\chi'_s(G) \le \tfrac{5}{4}\Delta(G)^2\), which would be best possible as demonstrated by the blow up of a 5-cycle.
 Andersen \cite{Andersen92} and independently Hor\'{a}k, Qing and Trotter \cite{HQT93} proved the conjecture for multigraphs of maximum degree at most 3.  
 (Kostochka {\it et al} \cite{KLRSWY16} showed that every planar multigraph with maximum degree at most 3 in fact has strong chromatic index at most 9, verifying a conjecture of Faudree {\it et al} in  \cite{FaudreeGyarfasSchelpTuza1990StrongChromaticIndex}.)
Molloy and Reed~\cite{MolloyReed1997} proved that \( \chi'_s(G) \le 1.998\,\Delta(G)^2 \) provided that $\Delta(G)$ is sufficiently large. Bruhn and Joos~\cite{BruhnJoos2018StrongerBoundStrongChromaticIndex} improved this bound to \( 1.93\,\Delta(G)^2 \) and commented that the method used in their work does not produce a bound better than $1.73\Delta(G)^2$. Subsequent improvements of the bound were made by Bonamy, Perrett, and Postle~\cite{BonamyPerrettPostle2022SparseNeighbourhoods} (to \( 1.835\,\Delta(G)^2 \))  and by Hurley, de Joannis de Verclos, and Kang~\cite{HurleyJoannisKang2022ImprovedProcedure} (to  \( 1.772\,\Delta(G)^2 \)). In his master's thesis, Davey \cite{davey2024local} further improved the bound to  \( 1.73\,\Delta(G)^2 \) for sufficiently large $\Delta(G)$, and also obtained good bounds on the strong chromatic index of bipartite graphs. (After we posted the first version of this paper on arXiv, we learned about Davey's thesis from Ross Kang, and that those results will be included  in a forthcoming paper by Davey, de Joannis de Verclos, Hurley, Kang, and Volec.) While bounding the strong chromatic index has proven to be challenging, related results have been established for the fractional strong chromatic index $\chi_{f,s}'(G)$ (which we do not formally define). For example, one can obtain strong bounds for $\chi_{f,s}'(G)$ by bounding the strong clique number (the clique number of $L(G)^2$), and the latter is studied by Cames van Batenburg in his Ph.D. thesis \cite{vanBatenurgThesis}.

There is also a bipartite version of this Erd\H{o}s-Ne\v{s}etr\v{e}l conjecture, due to Faudree, Gy\'arf\'as, Schelp, and Tuza~\cite{FaudreeGyarfasSchelpTuza1990StrongChromaticIndex},  which states that \( \chi'_s(G) \le \Delta(G)^2 \). 
Steger and Yu~\cite{StegerYu1993InducedMatchings} verified the Faudree--Gy\'arf\'as--Schelp--Tuza conjecture for the case $\Delta (G)=3$. 
Brualdi and Quinn Massey~\cite{BrualdiQuinnMassey1993} made the following stronger conjecture.

\begin{conjecture}[Brualdi--Quinn Massey Conjecture] \label{conj:BQ}
For any bipartite \( G  \) with partite sets $A$ and $B$, 
\(
\chi'_s(G) \le \Delta_A \Delta_B, 
\)
where $\Delta_A=\max\{d_G(a):a\in A\}$ and $\Delta_B=\max\{d_G(b): b\in B\}$. 
\end{conjecture}

Nakprasit~\cite{Nakprasit2008StrongChromaticIndexBipartite} verified the Brualdi--Quinn Massey Conjecture for \( \Delta_A = 2 \). Huang, Yu, and Zhou~\cite{HuangYuZhou2017StrongChromaticIndex3DeltaBipartite}, and independently, Bensmail, Lagoutte, and Valicov~\cite{BensmailLagoutteValicov2016StrongEdgeColoring3DeltaBipartite}, verified this conjecture for \( \Delta_A = 3 \).  Davey~\cite{davey2024local} showed that \(\chi_s'(G)\le 1.6632\Delta_A\Delta_B\) when $\Delta_B=p\Delta_A$ for $p\in\{0.1,0.2,\ldots,1\}$ and  \(\chi_s'(G)\le 1.6254\,\Delta(G)^2 \) for any bipartite graph with a large enough $\Delta(G)$.  We also mention that, Cames van Batenburg, Kang and Pirot \cite{van2020strong} used the fact that the strong clique number of such a bipartite graph is at most $\Delta_A\Delta_B$ to show that (1) $\chi_{f,s}'(G) \le 1.5\Delta(G)^2$ for any bipartite graph $G$, and (2) $\chi_{f,s}'(G) \le 1.625\Delta(G)^2$ for any triangle-free graph $G$.

In this paper, we prove the following, which only requires that $\Delta_A$ and $\Delta_B$ are large enough.  

\begin{theorem}\label{thm-main}
Let \( G\) be a bipartite graph with partite sets $A$ and $B$, let $\Delta_A=\max\{d_G(a):a\in A\}$ and $\Delta_B=\max\{d_G(b): b\in B\}$. Then  \(
\chi'_s(G) \le 1.676\,\Delta_A \Delta_B\) provided that  $\Delta_A$ and $\Delta_B$ are both sufficiently large.
\end{theorem}

We note that it suffices to prove Theorem~\ref{thm-main} for {\it biregular} bipartite graphs, i.e., bipartite graphs in which all vertices in the same partite set have the same degree. This is because of the following reason: If $G$ is a bipartite graph with partite sets $A$ and $B$, and if $\Delta_A=\max\{d_G(a):a\in A\}$ and $\Delta_B=\max\{d_G(b): b\in B\}$, then $G$ is an induced subgraph of some biregular bipartite graph whose vertices have degrees $\Delta_A$ or $\Delta_B$.

The main step in our proof of Theorem~\ref{thm-main} is to solve an extremal problem on biregular bipartite graphs. To convert the original coloring problem to a new extremal problem, we use the idea of Molloy and Reed~\cite{MolloyReed1997} (also used in~\cite{BruhnJoos2018StrongerBoundStrongChromaticIndex,BonamyPerrettPostle2022SparseNeighbourhoods,HurleyJoannisKang2022ImprovedProcedure}) that considers the density of the neighborhood of vertices in the line graph. 
It is well known that a strong edge-coloring of $G$ is equivalent to a proper vertex coloring of $L(G)^2$, the square of the line graph of $G$. To apply the Molloy--Reed approach to bound the chromatic number of $L(G)^2$, it is essential to bound the number of edges in the subgraph of $L(G)^2$ induced by the neighborhood of each vertex. 
Specifically, for any edge $e\in E(G)$, let
$$N_e^s:=\Big\{f\in E(G): d_G(e,f)\le 1\Big\},$$ 
which is the neighborhood of $e$ in $L(G)^2$. We aim to bound the number of pairs of edges $\{e_1,e_2\}\subseteq N_e^s$ such that $e_1$ and $e_2$ are adjacent in $L(G)^2$, which is formally defined below  for each $e\in E(G)$: 

\begin{equation*}
m_e:=\left|\left\{\{e_1,e_2\}\in {N_e^s\choose 2}:  d_G(e_1,e_2)\le 1\right\}\right|.
\end{equation*}
Note that, for any distinct $e_1, e_2 \in N_e^s$, $d_G(e_1,e_2)\le 1$ (i.e., $e_1e_2\in E(L(G)^2)$) if and only if  $e_1$ and $e_2$ are the two end edges of a path of length 2 or 3 in $G$. 

The remainder of the paper is organized as follows.
In Section \ref{section2}, we convert the problem of bounding $m_e$ to an extremal problem over the family of biregular bipartite graphs by counting paths of length 2 or 3, and provide an upper bound on $m_e$ in terms of the number of such paths. In Section \ref{section3}, we optimize the bound on $m_e$ obtained from Section \ref{section2}.   The main result,  Theorem~\ref{thm-main}, is then proved in Section \ref{section4} by using the optimized bound and a recent result of Hurley, de Joannis de Verclos, and Kang \cite{HurleyJoannisKang2022ImprovedProcedure}. 

\section{Counting paths in biregular bipartite graphs}\label{section2}

The goal of this section is to convert the problem of bounding $m_e$ to an extremal problem on counting certain paths of length 2 or 3 in biregular bipartite graphs. To facilitate counting, we will consider {\it ordered paths} of lengths 2 or 3, which will be denoted as sequences of distinct vertices of length 3 or 4 in which consecutive vertices are adjacent. For example, the two sequences $abcd$ and $dcba$ are considered different ordered paths.

To describe and solve that counting problem, we need to fix some notation. Let $\mathcal{H}$ be the family of biregular bipartite graphs with partite sets $A$ and $B$ such that the vertices in $A$ have degree  $\Delta_A$ and the vertices in $B$ have degree $\Delta_B$. For each $H\in \hh$, consider partitions  $A=A_1 \cup A_2$ and $B=B_1 \cup B_2$, such that $A_1\cap A_2=\emptyset$, $B_1\cap B_2=\emptyset$, $|A_1| =\Delta_B$, and $|B_1| =\Delta_A$. With respect to this partition, we define the following sets:  
\begin{itemize} 
 
\item $\mathcal P(H)$ denotes the set of all ordered paths $abcd$ of length 3 in $H$ with $ab\notin  E_H(A_2,B_2)$.\medskip
\item $\mathcal B(H)$ denotes the set of all ordered paths $abcd$ in $\mathcal P(H)$ such that $cd\in E_H(A_2, B_2)$.\medskip
\item $\mathcal C(H)$ denotes the set of all 4-cycles in $H-E_H(A_2, B_2)$. 
\end{itemize}
Thus $\mathcal P(H)\backslash \mathcal B(H)$ is the set of all ordered paths $abcd$ in $H$ with $ab,cd\in  E(H)\setminus E_H(A_2,B_2)$. We now bound $m_e$ in terms of $|\pp(H)|$, $|\bb(H)|$, and $|\cc(H)|$, as established in the following lemma (see Figure 1 for an illustration).

    \begin{figure}[!htb]
\begin{center}
 \begin{tikzpicture}[
    % scale 控制整体大小，xscale 控制左右拉伸，yscale 控制上下压缩
    scale=0.6, xscale=1, yscale=0.8, transform shape,
    % 椭圆集合样式
    set_ellipse/.style={draw, ellipse, minimum width=4.5cm, minimum height=1.3cm, thick},
    % 顶点样式
    vertex/.style={circle, fill=black, inner sep=2pt},
    % 侧边 A, B 标签样式
    side_label/.style={font=\huge\bfseries\itshape}
]

    % --- Internal elements of X (Box removed) ---
    % Coordinates are kept at (0,0) as the center reference
    
    % Left side A, B labels
    \node[side_label, xshift=-3.5cm, yshift=1.5cm] at (0,0) {A};
    \node[side_label, xshift=-3.5cm, yshift=-1.5cm] at (0,0) {B};

    % Ellipses
    \node[set_ellipse] (XA) at (0, 1.5) {};
    \node[above=0.1cm of XA, font=\large] {$A_1$};
    \node[set_ellipse] (XB) at (0, -1.5) {};
    \node[below=0.1cm of XB, font=\large] {$B_1$};

    % Vertices and connections
    \node[vertex, label=above:{$u$}] (a) at (-1.6, 1.4) {};
    \node[vertex] (a2) at (0.4, 1.4) {};
    \node[vertex] (a3) at (1.8, 1.4) {};
    \node at (1.1, 1.45) {\dots\dots};

    \node[vertex, label=below:{$v$}] (b) at (-1.6, -1.4) {};
    \node[vertex] (b2) at (0.4, -1.4) {};
    \node[vertex] (b3) at (1.8, -1.4) {};
    \node at (1.1, -1.45) {\dots\dots};

    % Edges
    \draw[thick] (a) -- (b) node[midway, left, black] {$e$};
    \draw (a) -- (b2);
    \draw (a) -- (b3);
    \draw (b) -- (a2);
    \draw (b) -- (a3);

    % --- Elements for Y ---
    \node[set_ellipse] (YA) at (7.5, 1.5) {};
    \node[above=0.1cm of YA, font=\large] {$A_2$};
    \node[set_ellipse] (YB) at (7.5, -1.5) {};
    \node[below=0.1cm of YB, font=\large] {$B_2$};

\end{tikzpicture}
    \caption{\label{fig1} $A_1$, $A_2$, $B_1$, $B_2$}
\end{center}
\end{figure}
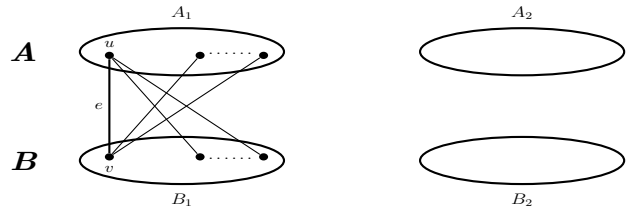\medskip

\begin{lemma} \label{lem me PBC}
Let  $H\in \hh$ be a biregular bipartite graph with partite sets $A$ and $B$, and let $\Delta_A, \Delta_B$ denote the degrees of vertices in $A,B$, respectively. For an edge $e = uv \in E(H)$ with $u\in A$ and $v\in B$, let $A_1 = N_H(v)$ and $B_1 = N_H(u)$, and let $A_2 = A \setminus A_1$ and $B_2 = B \setminus B_1$.
Let $N_e^s:=\left\{f\in E(G): d_G(e,f)\le 1\right\}$ and $m_e:=\left|\left\{\{e_1,e_2\}\in {N_e^s\choose 2}:  d_G(e_1,e_2)\le 1\right\}\right|$. Then   $$  2m_e\le |\pp(H)|-|\bb(H)|-4|\cc(H)|+ 2(\Delta_A^2\Delta_B+\Delta_A\Delta_B^2).$$
\end{lemma} 

\begin{proof}By definition, $N_e^s=E(H)\setminus E_H(A_2,B_2)$. For any two distinct edges $e_1, e_2 \in N_e^s$, they are adjacent in $L(H)^2$ (i.e., $d_H(e_1, e_2) \le 1$) if and only if  $e_1$ and $e_2$ are the two end edges of a path of length 2 or 3 in $H$.  To enumerate such $\{e_1,e_2\}$, we consider the following sets of ordered paths:
$$ \mathcal{P}_2 = \{abc : ab, bc \in N_e^s\} \quad \text{and} \quad \mathcal{P}_3 = \{abcd : ab, cd \in N_e^s\}. $$
 Each path in $\pp_2\cup \pp_3$ corresponds to an adjacent pair of edges in  $L(H)^2$. Specifically, any unordered pair $\{e_1, e_2\}$ with $d_H(e_1, e_2) \leq 1$ corresponds to exactly two ordered paths in $\pp_2\cup \pp_3$, except when $\{e_1,e_2\}$ is a matching contained in some $4$-cycle $C\in \cc(H)$; in this exceptional case, the pair $\{e_1, e_2\}$ is contained in four distinct ordered paths of length 3 in $\mathcal{P}_3$.
 %(which contains exactly 4 ordered paths between $e_1$ and $e_2$ in $\pp_3$). 
 Since each 4-cycle in $\cc(H)$ has two such matchings, we have 
%\begin{equation*}
%  2m_e = \left| \{ abcd : ab, cd \in N_e^s %\in E^1(H) 
%  \} \right| + \left| \{ abc : ab, bc \in N_e^s %E^1(H) 
%  \} \right|-4|\mathcal{C}(H)|.
%\end{equation*}
\begin{equation*}
  2m_e \le \left| \pp_3 \right| + \left| \pp_2 \right|-4|\mathcal{C}(H)|.
\end{equation*}

By definition, $\left|\pp_3\right|= |\mathcal P(H)| - |\mathcal B(H)|$. We now bound $|\pp_2|$
%next show that $\left| \pp_2\right|\le  2(\Delta_A+\Delta_B)\Delta_A\Delta_B$ 
by considering the location of the middle  vertex $b$ of the ordered paths $abc\in \pp_2$. 
\begin{itemize}
    \item First, the number of paths $abc$ in $\pp_2$ with $b\in A_1$ is at most $\Delta_B\Delta_A^2$, since when $b\in A_1$, there are $|A_1|=\Delta_B$ choices for $b$ and, for each choice of $b$, there are at most $2\binom{d_H(b)}{2}\le \Delta_A^2$ choices for $a$ and $c$.

    \item The number of paths $abc$ in $\pp_2$ with $b\in B_1$ is at most $\Delta_A\Delta_B^2$, since when $b\in B_1$, there are $|B_1|=\Delta_A$ choices for $b$ and,  
    for each choice of $b$, there are at most $2\binom{d_H(b)}{2}\le \Delta_B^2$ choices for $a$ and $c$.

    \item The number of paths $abc$ in $\pp_2$ with $b\in A_2$ is at most $\Delta_A^2\Delta_B$; since when $b\in A_2$ we have $a,c\in B_1$ and, hence, there are at most $2\binom{|B_1|}{2}\le \Delta_A^2$ choices for $a$ and $c$ and, for each choice of $(a,c)$  there are at most $\Delta_B$ choices for $b$.

    \item The number of paths $abc$ in $\pp_2$ with $b\in B_2$ is at most $\Delta_B^2\Delta_A$; since when $b\in B_2$ we have $a,c\in A_1$ and, hence, there are at most $2\binom{|A_1|}{2}\le \Delta_B^2$ choices for $a$ and $c$ and, for each choice of $(a,c)$ there are at most $\Delta_A$ choices for $b$.
\end{itemize} 
Therefore, $|\mathcal{P}_2|\le 2(\Delta_A^2 \Delta_B + \Delta_A\Delta_B^2)$. Hence, the assertion of the lemma holds.  
\end{proof}

The inequality established in Lemma~\ref{lem me PBC} effectively reduces the problem of bounding $m_e$, the number of edges in the neighborhood of a vertex in $L(G)^2$, to an extremal problem over the family $\hh$ of biregular bipartite graphs. Accordingly, we now proceed to bound the quantity $|\mathcal P(H)| - |\mathcal B(H)| - 4|\mathcal C(H)|$ for $H\in \hh$. 
\begin{comment}
Let $\mathcal{H}$ be the family of biregular graphs $H=(A,B;E)$ with $A=U \cup A_2$ and $B=V \cup B_2$, where $U\cap X=\emptyset$, $V\cap Y=\emptyset$, and $|U| =\Delta_B, |V| =\Delta_A$. 
For each $H\in \mathcal H$, let $\mathcal C(H)$ denote the set of all 4-cycles in $H[U\cup B_2]$ and $H[V\cup A_2]$.
We also consider ordered paths in $H$, which are denoted as sequences of distinct vertices (such as $abcd$) in which every two consecutive vertices form an edge.  Note that $abcd$ and $dcba$ are distinct ordered paths. 
We use $\mathcal P(H)$ to denote the set of all ordered paths $abcd$ in $H$ with $ab\notin  E_H(A_2,B_2)$, and use $\mathcal B(H)$ to denote the set of all ordered paths $abcd$ in $\mathcal P(H)$ such that $cd\in E_H(X, Y)$. Then $\mathcal P(H)\backslash \mathcal B(H)$ is the set of all ordered paths $abcd$ in $H$ with $ab,cd\notin  E_H(A_2,B_2)$.

We now proceed to bound  $|\mathcal P(H)| - |\mathcal B(H)| - 4|\mathcal C(H)|$ for $H\in \hh$. 
\end{comment}
%Assume for a contradiction that  $H=H(U\cup A_2,V\cup B_2)$ is a graph in $\hh$ with $ |\pp(H)|-|\bb(H)|-4|\cc(H)|\ge  \frac{9}{4}\Delta^4+ o(\Delta^4)$. 

In the remainder of this section, we fix the following notation. Let $H\in \hh$ be a biregular bipartite graph with partite sets $A$ and $B$ such that all vertices in $A$ have degree  $\Delta_A$ and all vertices in $B$ have degree $\Delta_B$. Let $A=A_1 \cup A_2$ and $B=B_1 \cup B_2$, such that $A_1\cap A_2=\emptyset$, $B_1\cap B_2=\emptyset$,  $|A_1| =\Delta_B$, and $|B_1| =\Delta_A$. Define $\gamma\in[0,1]$ by letting $|E_H(A_1,B_1)|=(1-\gamma)\Delta_A\Delta_B$.    We will bound $|\pp(H)|$, $|\cc(H)|$, $|\bb(H)|$ separately. First, we bound $|\pp(H)|$. 
  
\begin{lemma}\label{p(H)}
    $|\pp(H)|\le (2+2\gamma)\Delta_A^2\Delta_B^2 +o(\Delta_A^2\Delta_B^2)$.
\end{lemma}
\begin{proof}
    To form an ordered path $abcd\in \pp(H)$, we recall that $ab$ corresponds to an edge in $E(H)\setminus E_H(A_2,B_2)$; 
    so there are exactly $2|E(H)\setminus E_H(A_2,B_2)|$ choices for $ab$. Note that 
    \begin{align*}
        |E(H)\setminus E_H(A_2,B_2)|& =|E_H(A_1,B_2)|+|E_H(A_2,B_1)|+|E_H(A_1,B_1)|\\
        &=\sum_{v\in A_1}d_H(v)+\sum_{v\in B_1}d_H(v)-|E_H(A_1,B_1)|\\
        &=|A_1|\Delta_A+|B_1|\Delta_B-|E_H(A_1,B_1)|\\
        &=(1+\gamma)\Delta_A\Delta_B. 
        \end{align*}
    
    Now, fix a choice of the sequence $ab$. Since $abcd$ is a path, the number of choices for $c$ is at most $d_H(b)-1$, and the number of choices for $d$ is at most $d_H(c)-1$ for each choice of $c$. Since $H$ is bipartite, it follows that $a$ and $c$ must belong to the same partite set, and similarly $b$ and $d$ must belong to the same partite set. Given that one of $d_H(b)$ and $d_H(c)$  is bounded by $\Delta_A$ and the other by $\Delta_B$, there are at most $(\Delta_A-1)(\Delta_B-1)$ ways to extend $ab$ to the ordered path $abcd$.
    
   Therefore,
    $|\pp(H)|\le 2(1+\gamma)\Delta_A\Delta_B(\Delta_A-1)(\Delta_B-1)=(2+2\gamma)\Delta_A^2\Delta_B^2+O(\Delta_A^2\Delta_B+\Delta_A\Delta_B^2).$
\end{proof}

To bound $|\cc(H)|$ and $|\bb(H)|$, it is convenient to introduce the following notation. For $i\in [2]$,  let $d_i(x)=|N(x)\cap B_i|$ for $x\in A_2$, and $d_i(y)=|N(y)\cap A_i|$ for $y\in B_2$; so $d_1(x)+d_2(x)=\Delta_A$ and $d_1(y)+d_2(y)=\Delta_B$. 
For $i\in [3]$, let \[ S_{A,i}:=\sum_{x\in A_2}d_1(x)^i \qquad\text{and}\qquad S_{B,i}:=\sum_{y\in B_2}d_1(y)^i. \]

By double counting the edges between $A_2$ and $B_1$,  we have  
\begin{equation}\label{equ: Sx1}
    S_{A,1} =\sum_{x\in A_2}d_1(x)=|E_H(A_2,B_1)|=|B_1|\Delta_B-|E_H(A_1,B_1)|=\gamma \Delta_A\Delta_B
\end{equation}   
and by double counting the edges between $A_1$ and $B_2$, we have
\begin{equation}\label{equ: Sx2}
S_{B,1} =\sum_{y\in B_2}d_1(y)=|E_H(A_1,B_2)|=|A_1|\Delta_A-|E_H(A_1,B_1)|=\gamma \Delta_A\Delta_B.
\end{equation}

We now bound $|\cc(H)|$ from below.  

\begin{lemma}\label{clm C(H)}
\[
 |C(H)| \ge 
 \frac{ 1+o(1) }{4} \left(\frac{S_{A,2}^2}{\Delta_A^2} + \frac{S_{B,2}^2}{\Delta_B^2} + (1-\gamma)^2\Delta_B S_{A,2} + (1-\gamma)^2\Delta_A S_{B,2} + (1-\gamma)^4\Delta_A^2\Delta_B^2   \right)+o(\Delta_A^2\Delta_B^2) . \]
\end{lemma} 
\begin{proof} 
Let $\cc_{A_1,B}$ denote the set of 4-cycles contained in $H[A_1\cup B]$, and $\cc_{A_2,B_1}$ denote the set of 4-cycles  contained in $H[A_2\cup B_1]$. Then $\cc_{A_1,B}\cup \cc_{A_2,B_1}\subseteq \cc(H)$. Since $E_H(A_1,B)\cap E_H(A_2,B_1)=\emptyset$, we have $|\mathcal{C}(H)| \ge |\mathcal{C}_{A_1, B}| + |\mathcal{C}_{A_2, B_1}|$.  We proceed to bound $|\cc_{A_1,B}|$ and $|\cc_{A_2,B_1}|$ from below.

Observe that 
   \begin{align*}
   |\cc_{A_1,B}| &= \sum_{\{u_1, u_2\} \subseteq A_1} \binom{|N_H(u_1)\cap N_H(u_2)|}{2}\\
   &\ge \binom{\Delta_B}{2} \binom{ 
   \frac{1}{\binom{\Delta_B}{2}} \sum_{\{u_1,u_2\} \subseteq A_1} |N_H(u_1)\cap N_H(u_2)| }{2} \quad(\mbox{by Jensen's inequality as $|A_1|=\Delta_B$})\\
   &= \binom{\Delta_B}{2} \binom{ 
   \frac{1}{\binom{\Delta_B}{2}} \sum_{b \in B} \binom{|N_H(b)\cap A_1|}{2}}{2} \quad (\mbox{by double counting}).\\
   &=\frac{1+o(1)}{\Delta_B^2} \left(\sum_{b \in B} \binom{|N(b)\cap A_1|}{2} \right)^2 +o(\Delta_A^2\Delta_B^2)  .
   \end{align*}
   In the last equality, we used that
$\binom{X}{2}=\frac12(1+o(1))X^2$
if \(X\to\infty\), while otherwise the contribution is absorbed into the \(o(\Delta_A^2\Delta_B^2)\) error term.
 Note that    
 \begin{align*}
    \sum_{b \in B} \binom{|N(b)\cap A_1|}{2} &= \sum_{y\in B_2} \binom{d_1(y)}{2} + \sum_{v \in B_1} \binom{|N_H(v)\cap A_1|}{2}\\
        &=\frac{1}{2}\sum_{y\in B_2} d_1(y)^2 + \frac{1}{2}\sum_{v \in B_1} |N_H(v)\cap A_1|^2 -\frac{1}{2}\left(\sum_{y\in B_2} d_1(y)+\sum_{v \in B_1} |N_H(v)\cap A_1|\right)\\
        &\ge \frac{1}{2}\left(\sum_{y\in B_2} d_1(y)^2 + (1-\gamma)^2 \Delta_A\Delta_B^2 -\Delta_A\Delta_B\right)   \\ 
        &=\frac{1}{2}\left(S_{B,2}  + (1-\gamma)^2 \Delta_A\Delta_B^2 -\Delta_A\Delta_B\right),
    \end{align*}
    where the inequality holds because $\sum_{y\in B_2}d_1(y)=\gamma \Delta_A\Delta_B$ (by \eqref{equ: Sx1}),  $\sum_{v \in B_1} |N_H(v)\cap A_1| = |E_H(A_1,B_1)| = (1-\gamma) \Delta_A\Delta_B$, and $\sum_{v \in B_1} |N_H(v)\cap A_1|^2 \ge |B_1|\left(\frac{\sum_{v \in B_1} |N_H(v)\cap A_1|}{|B_1|}\right)^2= (1-\gamma)^2 \Delta_A\Delta_B^2$ (by Cauchy-Schwarz) as $|B_1|=\Delta_A$. 
    Hence
\begin{align*}
|\mathcal{C}_{A_1,B}|  \ge \frac{1+o(1)}{4\Delta_B^2}  \Big( S_{B,2} + (1-\gamma)^2\Delta_A\Delta_B^2 \Big)^2  +o(\Delta_A^2\Delta_B^2) 
 \label{eq4.5}
\end{align*}

 Next we bound $|\cc_{A_2,B_1}|$.  Observe that 
 \begin{align*}
   |\cc_{A_2,B_1}| &= \sum_{\{v_1, v_2\} \subseteq B_1} \binom{|N_H(v_1)\cap N_H(v_2)\cap A_2|}{2} \\
   &\ge \binom{\Delta_A}{2} \binom{ \frac{1}{\binom{\Delta_A}{2}} \sum_{\{v_1,v_2\}\subseteq B_1} \binom{|N_H(v_1)\cap N_H(v_2)\cap A_2|}{2} }{2} \quad (\mbox{by Jensen's inequality as $|B_1|=\Delta_A$})\\
   &=        
         \binom{\Delta_A}{2} \binom{ \frac{1}{\binom{\Delta_A}{2}} \sum_{x\in A_2} \binom{|N_H(x)\cap B_1|}{2} }{2} \quad (\mbox{by double counting})   \\
    &=   \binom{\Delta_A}{2} \binom{ \frac{1}{\binom{\Delta_A}{2}} \sum_{x\in A_2} \binom{d_1(x)}{2} }{2}  =\frac{1+o(1)}{\Delta_A^2} \Bigg( \sum_{x\in A_2} \binom{d_1(x)}{2}   \Bigg)^2 +o(\Delta_A^2\Delta_B^2)\\
    &= \frac{1+o(1)}{4\Delta_A^2} S_{A,2}^2 +o(\Delta_A^2\Delta_B^2) . 
     \end{align*}

Thus, we have 
\begin{align*}
    |\cc(H)|  \ge |\mathcal{C}_{A_1,B}|+ |\cc_{A_2,B_1}| 
    \ge \frac{ 1+o(1) }{4} \left(\frac{S_{A,2}^2}{\Delta_A^2} + \frac{ \left( S_{B,2} + (1-\gamma)^2\Delta_A\Delta_B^2 \right)^2 }{\Delta_B^2} \right)  +o(\Delta_A^2\Delta_B^2) 
\end{align*}

Similarly,  let $\cc_{A,B_1}$ denote the set of 4-cycles contained in $H[A\cup B_1]$, and $\cc_{A_1,B_2}$ denote the set of 4-cycles  contained in $H[A_1\cup B_2]$. We also have $|\cc(H)|\ge |\cc_{A,B_1}|+|\cc_{A_1,B_2}|$, 
and the same argument above shows that
\begin{itemize}
    \item $|\cc_{A,B_1}|\ge   \frac{1+o(1)}{4\Delta_A^2}  \left( S_{A,2} + (1-\gamma)^2\Delta_B\Delta_A^2 \right)^2+o(\Delta_A^2\Delta_B^2)  , $
\item  $|\cc_{A_1,B_2}| \ge \frac{1+o(1)}{4\Delta_B^2} S_{B,2}^2+o(\Delta_A^2\Delta_B^2) .  $
\end{itemize}
This implies 
\begin{align*} 
|\cc(H)| \ge |\cc_{A,B_1}|+|\cc_{A_1,B_2}| \ge \frac{ 1+o(1) }{4} \left(\frac{S_{B,2}^2}{\Delta_B^2} + \frac{ \left( S_{A,2} + (1-\gamma)^2\Delta_B\Delta_A^2 \right)^2 }{\Delta_A^2} \right) +o(\Delta_A^2\Delta_B^2) 
\end{align*}

    Now the assertion of the lemma holds by averaging the above two lower bounds for $|\cc(H)|$.
    \end{proof}
It remains to bound $|\bb(H)|$. Recall that $\bb(H)$ consists of ordered paths $abcd$ with $ab\notin E_H(A_2,B_2)$ and $cd\in E_H(A_2,B_2)$.
\begin{lemma}\label{clm BH} 
\[
  |B(H)| \ge    \frac{\Delta_B}{\Delta_A}S_{A,3}+\frac{\Delta_A}{\Delta_B} S_{B,3}- 3\Delta_B S_{A,2} -3\Delta_A S_{B,2} + 4\gamma \Delta_A^2\Delta_B^2
          +o(\Delta_A^2\Delta_B^2)   \]
\end{lemma}
\begin{proof} 
 Let $\bb_1(H)=\big\{abcd\in \bb(H): bc\notin E_H(A_2,B_2)\big\}$ and $\bb_2(H)=\big\{abcd\in \bb(H): bc\in E_H(A_2,B_2)\big\}$.  Then $|\bb(H)|=|\bb_1(H)|+|\bb_2(H)|$. We will bound $|\bb_1(H)|$ and $|\bb_2(H)|$. 

For $|\bb_1(H)|$, we count the ordered paths $abcd \in \mathcal{B}_1(H)$ based on the location of the vertex $c$. Recall that $ab\notin E_H(A_2,B_2)$ and $cd \in E_H(A_2, B_2)$. When $c \in A_2$, we have $d\in B_2$, $b\in B_1$, and $a\in A$. Hence, for each choice of $c\in A_2$, there are $d_1(c)d_2(c)$ choices for the ordered path $bcd$, and for each choice of such $bcd$, there are $\Delta_B -1$ choices for $a$ to form the ordered path $abcd$. Thus, the number of ordered paths $abcd$ in $\bb_1(H)$ with $c\in A_2$ is 
\begin{align*}
    \sum_{c\in A_2}d_1(c)d_2(c)(\Delta_B -1)
    &= \sum_{c\in A_2}d_1(c)(\Delta_A-d_1(c))(\Delta_B -1)\\
    &= (\Delta_B-1)\Delta_A  \sum_{c\in A_2}d_1(c)-(\Delta_B-1) \sum_{c\in A_2}d_1(c)^2\\
    &\ge  \Delta_A (\Delta_B-1) S_{A,1}- \Delta_B S_{A,2} \\
    &=%\overset{\eqref{equ: Sx1}}{=} 
    \gamma\Delta_A^2 \Delta_B^2  - \Delta_B S_{A,2} +o(\Delta_A^2 \Delta_B^2) \quad (\mbox{by } \eqref{equ: Sx1}).
\end{align*} 
Similarly, $c \in B_2$ implies $d\in A_2$, $b\in A_1$, and $a\in B$, and  the number of ordered paths $abcd$ in $\bb_1(H)$ with $c\in B_2$ is at least
 $\gamma\Delta_A^2 \Delta_B^2- \Delta_A S_{B,2}+o(\Delta_A^2 \Delta_B^2)$. Hence, 
%By relabeling $c$ to $x$ (when $c\in A_2$) or $y$ (when $c\in B_2$),  we have 
\[
    |\bb_1(H)|\ge 2\gamma\Delta_A^2 \Delta_B^2- \Delta_B S_{A,2}- \Delta_A S_{B,2}+o(\Delta_A^2 \Delta_B^2).
\] 
 
Next, we bound $|\bb_2(H)|$. For an ordered path $abcd\in \bb_2(H)$, we have  $bc,cd\in E_H(A_2,B_2)$ and $ab\notin E_H(A_2,B_2)$.
If $b\in A_2$ and $c\in B_2$ then $a\in B_1$ and $d\in A_2$; so for each such choice of $bc$, the number of  ordered paths $abcd$  in $\bb_2(H)$ with $b\in A_2$ and $c\in B_2$ is $d_1(b)(d_2(c)-1)=d_1(b)(\Delta_B-d_1(c)-1)$. If $b\in B_2$ and $c\in A_2$ then $a\in A_1$ and $d\in B_2$; so for each such choice of $bc$, the number of paths $abcd$ in $\bb_2(H)$ with $b\in B_2$ and  $c\in A_2$ is $d_1(b)(d_2(c)-1)=d_1(b)(\Delta_A-d_1(c)-1)$.
Hence, 
\[|\bb_2(H)|=\sum_{bc\in E(H), b\in A_2, c\in B_2}d_1(b)(\Delta_B-d_1(c)-1)+ \sum_{bc\in E(H), b\in B_2,c\in A_2}d_1(b)(\Delta_A-d_1(c)-1).\]
%Summing over all edges $xy \in E_H(A_2, B_2)$, we obtain
By renaming $b$ to $x$ and $c$ to $y$ (when $b\in A_2$) or $b$ to $y$ and $c$ to $x$ (when $b\in B_2$), we have  
    \begin{align*}\label{eqB2}
        |\bb_2(H)|
        &= \sum_{xy\in E(H), x\in A_2, y\in B_2}\Big(-2d_1(x)d_1(y)+ (\Delta_B -1 )d_1(x)+ (\Delta_A -1)d_1(y) \Big).
    \end{align*}
    Note that 
    $$\sum_{xy\in E(H), x\in A_2, y\in B_2} d_1(x)= \sum_{x\in A_2} d_1(x)d_2(x)=\sum_{x\in A_2}d_1(x)\Big(\Delta_A-d_1(x)\Big);$$
  %  \overset{\eqref{equ: Sx1}}{=} \gamma \Delta_A^2\Delta_B-S_{A,2}, $$ 
    hence
    \begin{align*}
     & \quad \sum_{xy\in E(H), x\in A_2,y\in B_2}(\Delta_B-1)d_1(x)\\
     &=(\Delta_B-1)\Big(\Delta_A\sum_{x\in A_2}d_1(x)-\sum_{x\in A_2}d_1(x)^2\Big)\\
     & =
     (\Delta_B -1 )( \gamma \Delta_A^2\Delta_B-S_{A,2}) \quad (\mbox{by } \eqref{equ: Sx1})\\
     & \ge \gamma \Delta_A^2\Delta_B^2-\Delta_B S_{A,2} -\gamma \Delta_A^2\Delta_B\\ &=\gamma \Delta_A^2\Delta_B^2-\Delta_B S_{A,2}+o(\Delta_A^2\Delta_B^2). 
   \end{align*}  
    Similarly, 
    $$\sum_{xy\in E(H),x\in A_2,y\in B_2}(\Delta_A-1) d_1(y)\ge  \gamma \Delta_A^2\Delta_B^2-\Delta_A S_{B,2}+o(\Delta_A^2\Delta_B^2) .$$
    Further note that 
     \begin{align*}
    &\quad \sum_{xy \in E(H),x\in A_2,y\in B_2} 2d_1(x)d_1(y) \\
    &=\Delta_A\Delta_B\sum_{xy \in E(H),x\in A_2,y\in B_2} 2\cdot \frac{d_1(x)}{\Delta_A}\cdot \frac{d_1(y)}{\Delta_B} \\   
    &\le \Delta_A\Delta_B\sum_{xy \in E(H),x\in A_2,y\in B_2}  \left(\left(\frac{d_1(x)}{\Delta_A}\right)^2+\left(\frac{d_1(y)}{\Delta_B}\right)^2\right)\\
    &=\frac{\Delta_B}{\Delta_A}\sum_{x\in A_2} d_1(x)^2(\Delta_A-d_1(x))+\frac{\Delta_A}{\Delta_B}\sum_{y\in B_2} d_1(y)^2(\Delta_B-d_1(y))\\
    &=-\frac{\Delta_B}{\Delta_A}\sum_{x\in A_2} d_1(x)^3 -\frac{\Delta_A}{\Delta_B} \sum_{y\in B_2} d_1(y)^3 + \Delta_B \sum_{x\in A_2} d_1(x)^2 + \Delta_A \sum_{y\in B_2} d_1(y)^2\\
     &=-\frac{\Delta_B}{\Delta_A}S_{A,3} -\frac{\Delta_A}{\Delta_B} S_{B,3} + \Delta_B S_{A,2} + \Delta_A S_{B,2} .
\end{align*}
Therefore,
    \begin{align*}
        |\bb_2(H)|&\ge    \frac{\Delta_B}{\Delta_A}S_{A,3}+\frac{\Delta_A}{\Delta_B} S_{B,3}- 2\Delta_B S_{A,2} -2\Delta_A S_{B,2} + 2\gamma \Delta_A^2\Delta_B^2
        +o(\Delta_A^2\Delta_B^2)
    \end{align*}
 
Hence, by combining the above bounds for $|\bb_1(H)|$ and $|\bb_2(H)|$, we have 
 \[
  |\bb_1(H)|+|\bb_2(H)|\ge  \frac{\Delta_B}{\Delta_A}S_{A,3}+\frac{\Delta_A}{\Delta_B} S_{B,3}- 3\Delta_B S_{A,2} -3\Delta_A S_{B,2} + 4\gamma \Delta_A^2\Delta_B^2
       +o(\Delta_A^2 \Delta_B^2). \]
Now the assertion of the lemma follows, since $|\bb(H)|=|\bb_1(H)|+|\bb_2(H)|$. 
\end{proof}

\section{Bounding $|\pp(H)|-|\bb(H)|-4|\cc(H)|$}\label{section3}
In this section, we find an upper bound for $|\pp(H)|-|\bb(H)|-4|\cc(H)|$, which, combined with Lemma~\ref{lem me PBC}, provides a bound on $m_e$ that we will use to prove Theorem~\ref{thm-main}. 
First, we prove a technical lemma. 

\begin{lemma}\label{gamma-bound} 
Let $U(\gamma)=\frac{2 + 10\gamma - 4\gamma^3 + 2\gamma^4 - \gamma^5}{2(1+\gamma)}$. Then $U(\gamma) <2.348$ when $0\le \gamma \le 1$. 
\end{lemma}
\begin{proof}
 We bound the global maximum of $U(\gamma)$ on the closed interval $[0, 1]$. Consider the first derivative
\[ U'(\gamma) = \frac{-(4\gamma^5 - \gamma^4 + 12\gamma^2 - 8)}{2(1+\gamma)^2}. \]
The critical points of $U(\gamma)$ correspond to the roots of the polynomial $P(\gamma) = 4\gamma^5 - \gamma^4 + 12\gamma^2 - 8$. Note that $P(0) = -8 <0$ and $P(1) = 4 - 1 + 12 - 8 = 7 > 0$. By the Intermediate Value Theorem,  $P(\gamma)$ has at least one root in the interval $(0, 1)$. Furthermore, the derivative $P'(\gamma) = 4\gamma(5\gamma^3 - \gamma^2 + 6)$ is strictly positive on $(0, 1)$, since $6 - \gamma^2 > 5$ when $\gamma\in (0,1)$. Consequently, $P(\gamma)$ is strictly increasing on the interval $[0, 1]$, guaranteeing that the root is unique. Evaluating $P(\gamma)$ at specific points yields:
\begin{equation*}
\begin{aligned}
P(0.7764) &=4(0.7764)^5 - (0.7764)^4 + 12(0.7764)^2 - 8 < -0.001 < 0, \\
P(0.7765) &= 4(0.7765)^5 - (0.7765)^4 + 12(0.7765)^2 - 8 > 0.001>0.
\end{aligned}
\end{equation*}
Thus, the unique real root $r^*$ lies in the interval $(0.7764, 0.7765)$. %{\color{blue}I got an estimate $r^*\approx 0.7764556$, in this case, the upper bound would be 2.346}. 
Because $P(\gamma)$ transitions from negative to positive at $r$, it follows that $U'(\gamma) > 0$ for $0\le \gamma < r$ and $U'(\gamma) < 0$ for $r<\gamma\le 1$. Therefore, $U(\gamma)$ achieves its absolute maximum on $[0, 1]$ at $r^*$.

We now estimate $U(r^*)$. Note \[ U(r) = \frac{(2 + 10r + 2r^4) - (4r^3 + r^5)}{2(1+r)}. \]
Because $r^* \in (0.7764, 0.7765)$ %{\color{blue}we can establish a strict upper bound for $U(r^*)$ by maximizing the positive terms in the numerator while minimizing both the subtracted terms and the denominator. This yields:} 
and $U(\gamma)$ is increasing, 
\[ U(r^*) < \frac{(2 + 10(0.7765) + 2(0.7765)^4) - (4(0.7764)^3 + (0.7764)^5)}{2(1+0.7764)}  < 2.348. \]
%Combining this with the boundary evaluations $U(0) = 1$ and $U(1) = 2.25$, 
It follows that $U(\gamma) \le 2.348$ for all $\gamma \in [0, 1]$.  
\end{proof}

We can now state and prove the following key lemma. Recall from Section 2 the definitions of  $\hh$ and $\pp(H), \cc(H),\bb(H)$ for $H\in \hh$.
\begin{lemma}\label{lem PBC upbd}
   Let $H\in \hh$ be a biregular bipartite graph with partite sets $A,B$, such that $A=A_1\cup A_2$ and $B=B_1\cup B_2$,  $|A_1|=\Delta_B$, $|B_1|=\Delta_A$, $A_1\cap  A_2=\emptyset$, $B_1\cap B_2=\emptyset$,  the vertices of $A$ all have degree $\Delta_A$, and the vertices of $B$ all have degree $\Delta_B$.   Then 
   $$|\mathcal{P}(H)| - |\mathcal{B}(H)| - 4|\mathcal{C}(H)|\le  2.348 \Delta_A^2\Delta_B^2 + o(\Delta_A^2\Delta_B^2)$$
\end{lemma}
\begin{proof}
By Lemma~\ref{p(H)}, $|\mathcal{P}(H)| \le (2+2\gamma)\Delta_A^2\Delta_B^2 + o(\Delta_A^2\Delta_B^2)$, where   $0 \le \gamma \le 1$ We will minimize the quantity $|\mathcal{B}(H)|+4|\mathcal{C}(H)|$, subject to the following degree constraints:
\[ \sum_{x\in A_2}d_1(x) = \sum_{y\in B_2}d_1(y) = \gamma \Delta_A\Delta_B.\]
By Lemmas \ref{clm C(H)} and \ref{clm BH}, we have
\begin{align*}
    4|\cc(H)|+|\bb(H)|
    \ge &\big(1+o(1)\big)\left(\frac{S_{A,2}^2}{\Delta_A^2} + \frac{S_{B,2}^2}{\Delta_B^2} + (1-\gamma)^2\Delta_B S_{A,2} + (1-\gamma)^2\Delta_A S_{B,2} + (1-\gamma)^4\Delta_A^2\Delta_B^2   \right)\\
    & +\big(1+o(1)\big)\left(  \frac{\Delta_B}{\Delta_A}S_{A,3}+\frac{\Delta_A}{\Delta_B} S_{B,3}- 3\Delta_B S_{A,2} -3\Delta_A S_{B,2} + 4\gamma \Delta_A^2\Delta_B^2
        \right) \\
    =&\big(1+o(1)\big)\left( F_A +  F_B +\big((1-\gamma)^4+4\gamma\big)\Delta_A^2\Delta_B^2    \right),
\end{align*} 
where
\begin{align*}
F_A&=\frac{\Delta_B}{\Delta_A}S_{A,3}+\frac{1}{\Delta_A^2}S_{A,2}^2+\big((1-\gamma)^2-3\big)\Delta_BS_{A,2},\\
F_B&=\frac{\Delta_A}{\Delta_B}S_{B,3}+\frac{1}{\Delta_B^2}S_{B,2}^2+\big( (1-\gamma)^2-3\big)\Delta_AS_{B,2}.
\end{align*}

If $\gamma=0$ then by \eqref{equ: Sx1} and \eqref{equ: Sx2}, $d_1(x)=0$ for all $x\in A_2\cup B_2$; so $S_{A,i}=A_{B,i}=0$ for $i\in [3]$ and, hence,
%all the degrees occurring in $S_{A,i}$ and $S_{B,i}$ are zero for $i\in [3]$. 
\[
|\mathcal P(H)|-|\mathcal B(H)|-4|\mathcal C(H)|
\le \Delta_A^2\Delta_B^2+o(\Delta_A^2\Delta_B^2),
\]
as desired. 

We may therefore assume that $\gamma>0$. %{\color{blue} Let $u_x = d_1(x)^{1/2}$ and $v_x = d_1(x)^{3/2}$ so that $S_{A,2} = \sum_{x \in A_2} u_x v_x$.} 
By the Cauchy-Schwarz inequality and \eqref{equ: Sx1}, we have $$S_{A,2}^2 \le %{\color{blue}\left( \sum_{x \in A_2} u_x^2 \right) \left( \sum_{x \in A_2} v_x^2 \right)} = 
S_{A,1}S_{A,3} = \gamma\Delta_A\Delta_BS_{A,3}.$$
Thus. 
\begin{align*}
F_A
&\ge \frac{\Delta_B}{\Delta_A}\cdot \frac{S_{A,2}^2}{\gamma\Delta_A\Delta_B}+\frac{1}{\Delta_A^2}S_{A,2}^2+\big( (1-\gamma)^2-3\big)\Delta_BS_{A,2}\\
& = \frac{1+\gamma}{\gamma} \frac{1}{\Delta_A^2} S_{A,2}^2 +(\gamma^2-2\gamma-2)\Delta_BS_{A,2}\\
& = \frac{1+\gamma}{\gamma} \frac{1}{\Delta_A^2}\left( S_{A,2} - \frac{ \gamma(2+2\gamma-\gamma^2) }{ 2(1+\gamma) } \Delta_A^2\Delta_B \right)^2  - \frac{ \gamma(2+2\gamma-\gamma^2)^2 }{ 4(1+\gamma) } \Delta_A^2\Delta_B^2\\
&\ge - \frac{ \gamma(2+2\gamma-\gamma^2)^2 }{ 4(1+\gamma) } \Delta_A^2\Delta_B^2.
\end{align*}
Applying the same argument to $F_B$, we can show that 
\begin{align*}
F_B  \ge - \frac{ \gamma(2+2\gamma-\gamma^2)^2 }{ 4(1+\gamma) } \Delta_A^2\Delta_B^2. 
\end{align*}
Hence
\begin{align*}
|\mathcal P(H)|-|\mathcal B(H)|-4|\mathcal C(H)|
&\le \big(1+o(1)\big)\left(2+2\gamma-(1-\gamma)^4-4\gamma
 +\frac{ \gamma(2+2\gamma-\gamma^2)^2 }{ 2(1+\gamma) }  \right)\Delta_A^2\Delta_B^2\\
&=\big(1+o(1)\big)U(\gamma)\Delta_A^2\Delta_B^2.
\end{align*}
By Lemma~\ref{gamma-bound}, $U(\gamma)<2.348$ for all $\gamma \in [0, 1]$. Therefore, the assertion of the lemma holds.
\end{proof}

\section{Proof of Theorem~\ref{thm-main}}\label{section4}

First, we state a result  from Hurley, de Joannis de Verclos and Kang \cite{HurleyJoannisKang2022ImprovedProcedure}. %Given $\sigma >0$ and positive integer $D$, a graph $G$ is said to be {\it  $\sigma D$-sparse} if for every $v\in V(G)$ the subgraph $G[N(v)]$ induced by the neighborhood $N(v)$ of $v$ has at most $(1-\sigma)\binom{D}{2}$ edges. We now state a slightly generalized result of Theorem 1.2 in~\cite{HurleyJoannisKang2022ImprovedProcedure}. 
%\begin{theorem} [Hurley, de Joannis de Verclos, and Kang] \label{thm:HVK} Define $\epsilon: =\epsilon(\sigma) = \sigma/2 -\sigma^{3/2}/6$. For each $\iota>0$ and $0 < \sigma \le 1$, there is $D_{1,2}: =D_{1,2}(\iota)$ such that the chromatic number satisfies $$\chi(G) \le (1 -\epsilon+\iota)D$$ for every $\sigma D$-sparse graph $G$ with $D \ge D_{1,2}$. 
%\end{theorem}
Given $\sigma >0$, a graph $G$ is said to be {\it  $\sigma$-sparse} if for every $v\in V(G)$ the subgraph induced by the neighborhood of $v$ has at most $(1-\sigma)\binom{\Delta(G)}{2}$ edges. 
%We now state a slightly generalized result of Theorem 1.2 in~\cite{HurleyJoannisKang2022ImprovedProcedure}. 

\begin{theorem} [Hurley, de Joannis de Verclos, and Kang] \label{thm:HVK} Define $\epsilon: =\epsilon(\sigma) = \sigma/2 -\sigma^{3/2}/6$. For each $\iota>0$ and $0 < \sigma \le 1$, there exists $\Delta_{0}: =\Delta_{0}(\iota)$ such that the chromatic number satisfies $$\chi(G) \le (1 -\epsilon+\iota)\Delta(G)$$ for every $\sigma$-sparse graph $G$ with $\Delta(G) \ge \Delta_{0}$. 
\end{theorem}

We can now complete the proof of Theorem~\ref{thm-main}. Let $G$ be a bipartite graph with partite sets $A$ and $B$, and let $\Delta_A=\max \{d_G(a): a\in A\}$ and $\Delta_B=\max\{d_G(b):b\in B\}$. We may assume that $G$ is biregular with degrees $\Delta_A$ and $\Delta_B$. Then 
$$\Delta(L(G)^2)=(\Delta_A-1)(\Delta_B-1)+(\Delta_A-1)(\Delta_B-1)=2\Delta_A\Delta_B-2(\Delta_A+\Delta_B)+2.$$  
By Lemmas~\ref{lem me PBC} and \ref{lem PBC upbd}, we have, for each $e\in E(G)$, 
\begin{equation*}\label{ine6}
    m_e\le \big(1.174+o(1)\big)\Delta_A^2\Delta_B^2=\big(0.587+o(1)\big){\Delta(L(G)^2)\choose 2}. 
\end{equation*}
Hence, for each $v\in V(L(G)^2)$, the neighborhood of $v$ in $L(G)^2$ has at most $(1-\sigma){\Delta(L(G)^2)\choose 2}$ edges, where $\sigma=0.413$, i.e., $L(G)^2$ is $\sigma$-sparse. Applying Theorem~\ref{thm:HVK} with  $\epsilon=0.162$ and  very small positive real number $\iota$, we obtain, for sufficiently large $\Delta_A$ and $\Delta_B$,   $\chi'_s(G)=\chi(L(G)^2) \le 2(1-\epsilon)\Delta_A\Delta_B=1.676\Delta_A\Delta_B$.
%Thus, letting $D =2\Delta_A\Delta_B$ and $\sigma=0.413$, we see that $L(G)^2$ is $\sigma D$-sparse graph. Applying Theorem~\ref{thm:HVK} with the corresponding values of $\epsilon=0.162$ as above with very small positive real number $\iota$, for sufficiently large $\Delta_A\Delta_B$,  we get $\chi'_s(G)=\chi(L(G)^2) \le 2(1-\epsilon)\Delta_A\Delta_B=1.676\Delta_A\Delta_B$. 

\section{Acknowledgment} We would like to thank Ross Kang for bringing reference \cite{davey2024local} to our attention and informing us of his forthcoming paper  (now available on arXiv~\cite{davey2026}) on strong edge coloring with Davey, de Joannis de Verclos, Hurley, and Volec. We would also like to thank Wouter Cames van Batenburg for pointing us to the concept of strong clique number of graphs and for insightful comments and suggestions which led to this improved version and in particular the current version of Lemma 2.3.

\newpage
 \bibliographystyle{amsabbrv}
\bibliography{ref}
\end{document}